\documentclass{amsart}

\usepackage{amssymb}
\usepackage{amsfonts}
\usepackage{amsmath}
\usepackage{amsthm}
\usepackage{url}
\usepackage{color}
\usepackage[small,nohug,heads=LaTeX]{diagrams}
\diagramstyle[labelstyle=\scriptstyle]
\usepackage{graphicx}

\newcommand{\bi}{\begin{itemize}}
\newcommand{\ei}{\end{itemize}}
\newcommand{\bd}{\begin{description}}
\newcommand{\ed}{\end{description}}
\newcommand{\ben}{\begin{enumerate}}
\newcommand{\een}{\end{enumerate}}
\newcommand{\bbm}{\begin{bmatrix}}
\newcommand{\ebm}{\end{bmatrix}}
\newcommand{\bea}{\begin{eqnarray*}}
\newcommand{\eea}{\end{eqnarray*}}

\newtheorem{theorem}{Theorem}[section]

\newtheorem{conjecture}[theorem]{Conjecture}

\def\TT{\mathbb{T}}
\def\Xx{\mathcal{X}}
\def\kk{\mathbb{K}}
\def\KK{\mathbb{K}}
\def\Ff{\mathbb{F}}
\def\LL{\mathbb{L}}

\def\uu{\mathbb{U}}
\def\Tr{{\rm Tr}}
\def\GF{\mathbb{F}}

\def\2G2{\ensuremath{^2{\rm G}_2}}
\def\sl{{\rm{SL}}}
\def\gl{{\rm{GL}}}

\def\su{{\rm{SU}}}
\def\psl{{\rm{PSL}}}

\def\so{{\rm{SO}}}

\def\ppom{{\rm{(P) \Omega }}}
\def\om{{\rm{\Omega }}}
\def\sp{{\rm{Sp}}}

\def\ppsl{ ( {\rm{P}} ) {\rm{SL}}}
\def\ppsp{ ( {\rm{P}} ) {\rm{Sp}}}

\def\ppone{p \equiv 1 \bmod 4}
\def\pmone{p \equiv -1 \bmod 4}
\def\qpone{q \equiv 1  \bmod 4}

\definecolor{darkgreen}{rgb}{0,0.6,0}

\begin{document}

\title{Fifty Shades of Black}

\author{Alexandre Borovik}
\address{School of Mathematics, University of Manchester, UK; alexandre.borovik@gmail.com}
 \author{\c{S}\"{u}kr\"{u} Yal\c{c}\i nkaya}
 \address{Nesin Mathematics Village, Izmir, Turkey; sukru.yalcinkaya@gmail.com}


\begin{abstract}
The paper proposes a new and systematic approach to the so-called black box group methods in computational group theory. As the starting point of our programme, we construct Frobenius maps on black box groups of untwisted Lie type in odd characteristic and then apply them to
black box groups $X$ encrypting groups $\ppsl_2(q)$ in small odd characteristics. We propose an algorithm  constructing a black box field $\KK$ isomorphic to $\Ff_q$, and  an isomorphism from $\ppsl_2(\KK)$ to $X$.  The algorithm runs in time quadratic in the characteristic of the underlying field and polynomial in $\log q$.

Due to the nature of our work we also have to discuss a few methodological issues of the black box group theory.
\end{abstract}

\maketitle

\begin{footnotesize}
\tableofcontents
\end{footnotesize}

\section{Introduction}

Black box groups were introduced by Babai and Szemeredi \cite{babai84.229} as an idealized setting for randomized algorithms for solving permutation and matrix group problems in computational group theory. This paper belongs to a series of works aimed at development of systematic structural analysis of black box groups \cite{suko03,suko12A,suko12E,suko12B,suko12C,suko02,yalcinkaya07.171}.

The principal results of this paper are concerned with construction of Frobenius maps on black box Chevalley groups of untwisted type and odd characteristic, they are stated and proven in Section~\ref{ssec:Frobenius}.

In Section~\ref{sec:proofpsl2odd}, these constructions are applied to prove Theorem~\ref{cons:psl2} concerned with recognition of black box groups $\ppsl_2(q)$ for $\qpone$ and $q=p^k$ for some $k\geqslant 1$.

Our approach requires a detailed discussion of some methodological issues of black box group theory; this discussion is spread all over the paper and is supported by some ``toy''  mathematical results, such as Theorem~\ref{th:sl(2,2n)} that provides recognition of black box groups $SL_2(2^n)$ under a (rather hypothetical) assumption that  we are given an involution in the group.

\section{Black box groups and their automorphisms}

\subsection{Axioms  BB1 -- BB3}

A black box group $X$ is a black box (or an oracle, or a device, or an algorithm) operating with $0$--$1$ strings of bounded length which encrypt (not necessarily in a unique way) elements of some finite group $G$ (in various classes of black box problems the isomorphism type of $G$ could be known in advance or unknown). The functionality of a black box is specified by the following axioms, where every operation is carried out in time polynomial in terms of  $\log |G|$.

\begin{itemize}
\item[\textbf{BB1}] $X$ produces strings of fixed length $l(X)$ encrypting random (almost) uniformly distributed elements from $G$; the string length $l(X)$ is polynomially bounded in terms of $\log |G|$.
\item[\textbf{BB2}] $X$ computes, in time polynomial in $l(X)$, a string encrypting the product of two group elements given by strings or a string encrypting the inverse of an element given by a string.
\item[\textbf{BB3}] $X$ compares, in time polynomial in $l(X)$, whether two strings encrypt the same element in $G$---therefore identification of strings is a canonical projection
\begin{diagram} X & \rDotsto^\pi & G.
 \end{diagram}
\end{itemize}

We shall say in this situation that $X$ is a \emph{black box over $G$} or that a black box $X$ \emph{encrypts} the group $G$. Notice that we are not making any assumptions on computability of the projection $\pi$.

A typical example of a black box group is provided by a group $G$ generated in a big matrix group $\gl_n(r^k)$ by several matrices $g_1,\dots, g_l$. The product replacement algorithm \cite{celler95.4931} produces  a sample of (almost) independent elements from a distribution on $G$ which is close to the uniform distribution (see a discussion and further development in \cite{babai00.627, babai04.215, bratus99.91, gamburd06.411, lubotzky01.347, pak00.476, pak01.301, pak.01.476, pak02.1891}). We can, of course, multiply, invert, compare matrices. Therefore the computer routines for these operations together with the sampling of the product replacement algorithm run on the  tuple of generators $(g_1,\dots, g_l)$ can be viewed as a black box $X$ encrypting the group $G$. The group $G$ could be unknown---in which case we are interested in its isomorphism type---or it could be known, as it happens in a variety of other black box problems.

\subsection{Global exponent and Axiom BB4}

Notice that even in routine examples the number of elements of a matrix group $G$ could be astronomical, thus making many natural questions about the black box $X$ over $G$---for example, finding the isomorphism type or the order of $G$---inaccessible for all known deterministic methods. Even when $G$ is cyclic and thus is characterized by its order, existing approaches to finding multiplicative orders of matrices over finite fields are conditional and involve oracles either for the discrete logarithm problem in finite fields or for prime factorization  of integers.

Nevertheless black box problems for matrix groups have a feature which makes them more accessible:

\begin{itemize}
\item[\textbf{BB4}] We are given a \emph{global exponent} of $X$, that is, a natural number $E$ such that it is expected that $\pi(x)^E = 1$ for all strings $x \in X$ while computation of $x^E$ is computationally feasible (say, $\log E$ is polynomially bounded in terms of $\log |G|$).
\end{itemize}

Usually, for a black box group $X$ arising from a subgroup in the ambient group $\gl_n(r^k)$, the exponent of $\gl_n(r^k)$ can be taken for a global exponent of $X$.

One of the reasons why  the axioms BB1--BB4, and, in particular, the concept of global exponent, appear to be natural, is provided by some surprising model-theoretic analogies. For example, D'Aquino and Macintyre  \cite{D'Auquino00.311} studied non-standard finite fields defined in a certain fragment of bounded Peano arithmetic; it is called $I\Delta_0+\Omega_1$ and imitates proofs and computations of polynomial time complexity in modular arithmetic. It appears that such basic and fundamental fact as the Fermat Little Theorem has no proof which can be encoded in $I\Delta_0+\Omega_1$; the best that had so far been proven in $I\Delta_0+\Omega_1$ is that the multiplicative group $\GF_p^*$ of the prime field $\GF_p$ has a global exponent $E <2p$ \cite{D'Auquino00.311}.

We shall discuss model theory and logic connections of black box group theory in some details elsewhere.

\subsection{Relations with other black box groups projects}
\label{ssec:matrix-project}

\begin{quote}
\textbf{\emph{In this paper, we assume that all our black box groups satisfy assumptions BB1--BB4. }}
\end{quote}

We emphasize that we do not assume that black box groups under consideration in this paper are given as subgroups of ambient matrix groups; thus our approach is wider than the setup of the computational matrix group project \cite{leedham-green01.229}. Notice that we are not using the Discrete Logarithm Oracles for finite fields $\Ff_q$: in our original setup, we do not have fields. Nevertheless we are frequently concerned with black box groups encrypting classical linear groups; even so, some of our results (such as Theorems~\ref{th:sun-in-sln} and \ref{th:G2-in-SL8}) do not even involve the assumption that we know the underlying field of the group but instead assume the knowledge of the characteristic of the field without imposing bounds on the size of the field. Finally, in the case of groups over fields of small characteristics we can prove much sharper results, see, for example, Theorem~\ref{cons:psl2}. Here, it is natural to call characteristic $p$ ``small'', if it is known and if a linear or quadratic dependency of the running time of algorithm on $p$ does not cause trouble and algorithms are computationally feasible.

So we attach to statements of our results one of the two labels:
\bi
\item Known characteristic,
\item Small characteristic.
\ei
Our next paper \cite{suko12B} is dominated by ``known characteristic'' results. In this one, we concentrate on black box groups of known or small characteristics.

\subsection{Morphisms}
\label{sec:morphisms}

Given two  black boxes $X$ and $Y$ encrypting  finite groups $G$ and  $H$, correspondingly, we say that a map $\alpha$ which assigns strings from $Y$ to strings from $X$  is a \emph{morphism} of black box groups, if
\bi
\item the map $\alpha$ is computable in probabilistic time polynomial in $l(X)$ and $l(Y)$, and
\item there is an abstract homomorphism $\beta:G \to H$ such that the following diagram  is commutative:
\begin{diagram}
X &\rTo^{\alpha} &Y\\
\dDotsto_{\pi_{X}} & &\dDotsto_{\pi_{Y}}\\
G &\rTo^{\beta} & H
\end{diagram}
where $\pi_X$ and $\pi_Y$ are the canonical projections of $X$ and $Y$ onto $G$ and $H$, correspondingly.
\ei
We shall say in this situation that a morphism $\alpha$ \emph{encrypts} the homomorphism $\beta$. {For example, morphisms arise naturally when we replace a generating set for black box group $X$ by a more convenient one and start sampling the product replacement algorithm for the new generating set; in fact, we replace a black box for $X$  and deal with a morphism $Y \longrightarrow X$ from the new black box into $X$. Also, a black box subgroup $Z$ of $X$ is a morphism $Z\hookrightarrow X$.}

Slightly abusing terminology, we say that a morphism $\alpha$ is an embedding, or an epimorphism, etc., if  $\beta$ has these properties. In accordance with standard conventions, hooked arrows \begin{diagram}& \rInto & \end{diagram} stand for embeddings and doubleheaded arrows \begin{diagram}& \rOnto & \end{diagram} for epimorphisms; dotted arrows are reserved for abstract homomorphisms, including natural projections
\begin{diagram}
X & \rDotsto^{\pi_X} & \pi(X);
\end{diagram}
the latter are not necessarily morphisms, since, by the very nature of black box problems, we do not have efficient procedures for constructing the projection of a black box onto the (abstract) group it encrypts.

We further discuss morphisms in Sections \ref{50} and \ref{ssec:structure-recovery}.

\subsection{Shades of black}

Polynomial time complexity is an asymptotic concept, to work with it we need an infinite class of objects. Therefore our theory refers to some infinite family $\Xx$ of black box groups ($\Xx$ of course varies from one black box problem to another). For $X \in \Xx$, we denote by $l(X)$ the length of $0$--$1$ strings representing elements in $X$. We assume that, for every $X\in \Xx$, basic operations of generating, multiplying, comparing strings in $X$ can be done in probabilistic polynomial time in $l(X)$. We assume that encryption of group elements in $X$ is sufficiently economical and $l(X)$ is bounded by a polynomial in $\log |\pi(X)|$.

We also assume that the lengths $\log E(X)$ of  global exponents $E(X)$ for $X\in \Xx$ are bounded by a polynomial  in $l(X)$.

{Morphism $X \longrightarrow Y$ in $\Xx$ are understood as defined in Section~\ref{sec:morphisms} and their running times are bounded by a polynomial in $l(X)$ and $l(Y)$.

At the expense of {slightly increasing $\Xx$ and its bounds for complexity, we can include in $\Xx$ a collection of explicitly given ``known'' finite groups. Indeed, using standard computer implementations of finite field arithmetic, we can represent every group $Y = \gl_n(\GF_{p^k})$ as an algorithm or computer routine operating on $0$--$1$ strings of length $l(Y)= n^2k\log p$. Using standard matrix representations for simple algebraic groups, we can represent every group of points $Y = {\rm G}(\GF_{p^k})$ of a reductive algebraic group $\rm G$ defined over $\GF_{p^k}$ as a black box $Y$ generating and processing strings of length $l(Y)$ polynomial in $\log |\GF_{p^k}|$ and the Lie rank of $Y$. Therefore an ``explicitly defined'' group can be seen a black box group, perhaps of a lighter shade of black.

We shall use direct products of black boxes: if $X$ encrypts $G$ and $Y$ encrypts $H$ then the black box  $X\times Y$ generates pairs of strings $(x,y)$ by sampling $X$ and $Y$ independently, with operations carried out componentwise in $X$ and $Y$; of course, $X\times Y$ encrypts $G\times H$.

\begin{figure}[htbp]
\begin{center}
\vspace{.5in}
\includegraphics[scale=.15]{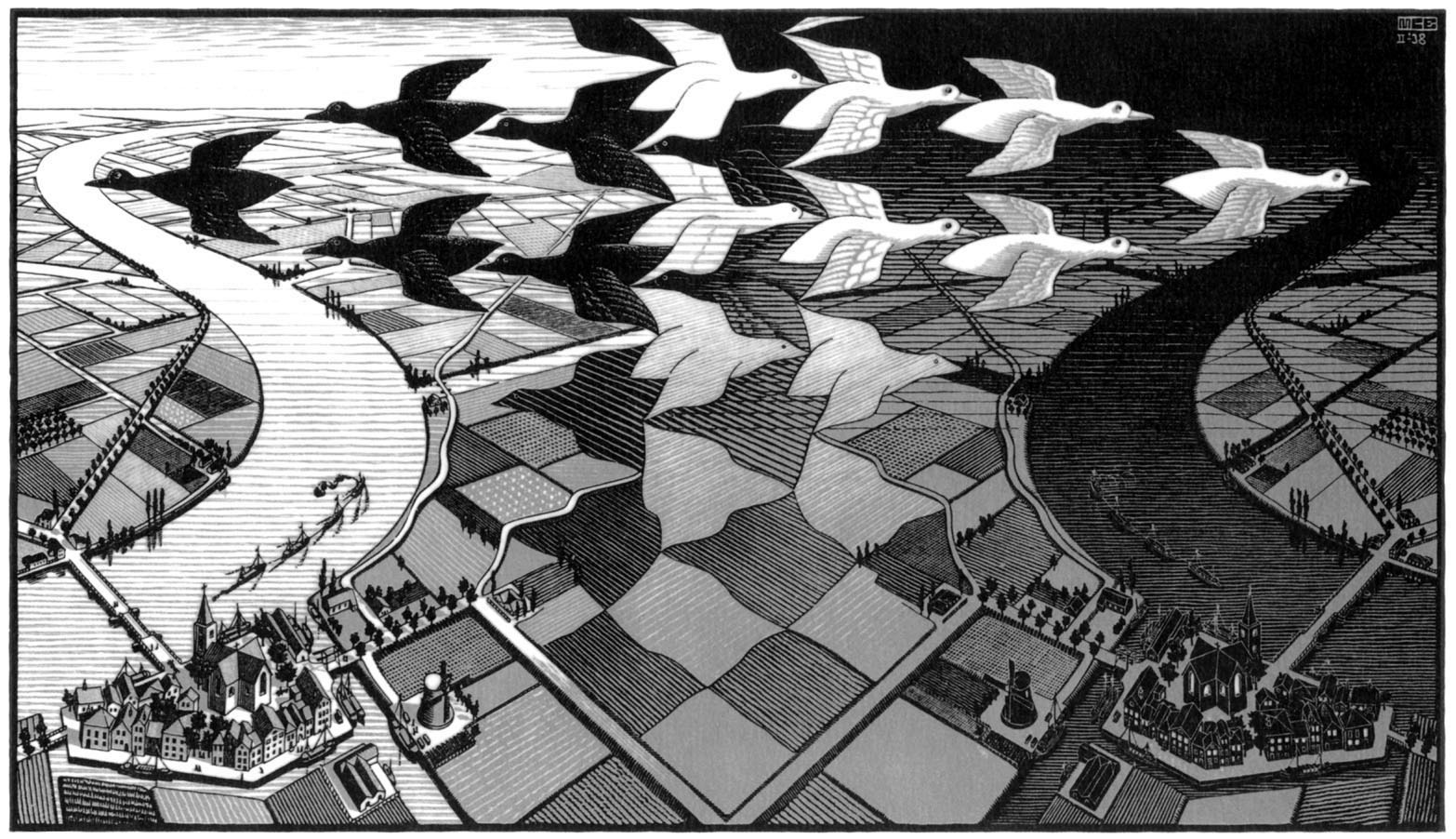}
\caption{ M.C. Escher,  \emph{Day and Night}, 1938}
\label{Escher}
\end{center}
\end{figure}

We  feel that the best way to understand a black box group
\begin{diagram}
G & \lDotsto & X
\end{diagram}
is a step-by-step construction of a chain of morphisms
\begin{diagram}
G & \lDotsto & X & \lOnto & X_1 & \lOnto & X_2 & \lOnto & \cdots & \lOnto  & X_n &\lOnto &G
\end{diagram}
at each step changing the shade of black and increasing amount of information provided by black boxes $X_i$.

Even in relatively simple black box problems we may end up dealing with a sophisticated category of black boxes and their  morphisms. Step-by-step transformation of black boxes into ``white boxes'' and their complex entanglement is captured well by Escher's famous woodcut, Figure~\ref{Escher}.

\subsection{Automorphisms as lighter shades of black}\label{50}

The first application of the ``shadows of black'' philosophy is the following self-evident theorem which explains how an automorphism of a group can be added to a black box encrypting this group.

\begin{theorem}\label{th:automorphism}
Let $X$ be a black box group encrypting a finite group $G$ and assume that each of\/ $k$ tuples of strings
\[
\tilde{x}^{(i)} = (x^{(i)}_1,\dots,x^{(i)}_m),\quad i = 1,\dots,k,
\]
generate $X$ in the sense that the projections $\pi\left(x^{(i)}_1\right),\dots,\pi\left(x^{(i)}_m)\right)$ generate $G$. Assume that the map
\[
\pi: x^{(i)}_j \mapsto \pi(x^{(i+1 \bmod k)}_j),\quad i = 0,\dots, k-1,\quad j =1,\dots, m,
\]
can be extended to an automorphism $a \in {\rm Aut}\, G$ of order $k$.
The black box group $Y$ generated in $X^k$ by the strings
\[
\bar{x}_j = \left(x^{(0)}_j, x^{(1)}_j, \dots, x^{(k-1)}_j\right), \quad j = 1,\dots, m,
\]
encrypts $G$ via the canonical projection on the first component
\[
(y_0,\dots, y_{k-1}) \mapsto \pi(y_o),
\]
and possess an additional unary operation, cyclic shift
\bea
\alpha: Y &\longrightarrow& Y\\
(y_0,y_2,\dots,y_{k-1},y_{k-1}) &\mapsto & (y_1,y_2,\dots,y_{k-1},y_0)
\eea
which encrypts the automorphism $a$ of $G$ in the sense that the following diagram commutes:
\begin{diagram}
Y &\rOnto^\alpha & Y\\
\dDotsto &&  \dDotsto\\
G & \rDotsto^a & G
\end{diagram}
\end{theorem}

A somewhat more precise  formulation of Theorem~\ref{th:automorphism} is that we can construct, in polynomial in $k$ and $m$ time, a commutative diagram
\begin{equation}
 \label{diag:Fr}
\begin{diagram}
 X & \lOnto^{\{\,\pi_i\,\}_{0\leqslant i\leqslant k-1}} & X^k & \lInto^\delta & Y&\rDotsto^\alpha& Y\\
 \dDotsto &&\dDotsto && \dDotsto && \dDotsto \\
G & \lDotsto^{\{\,p_i\,\}_{1\leqslant 0\leqslant k-1}} & G^k &\lDotsto^d& G &\rDotsto^a &G
\end{diagram}
\end{equation}
where  $d$ is the twisted diagonal embedding
\bea
d: G & \longrightarrow & G^k \\
 x & \mapsto & (x,x^a,x^{a^2},x^{a^{k-1}}),
\eea
and $p_i$ is the projection
\bea
p_i: G^k & \longrightarrow & G \\
(g_0,\dots,g_i,\dots,g_{k-1}) &\mapsto& g_i.
\eea

Of course, this construction leads to memory requirements increasing by factor of $k$, but, as our subsequent papers \cite{suko12B,suko12C} show, this is price worth paying. After all, in most practical problems the value of $k$ is not that big, in most interesting cases $k=2$.

A useful special case of Theorem~\ref{th:automorphism} is the following result about amalgamation of black box automorphisms, stated here in an informal wording rather than expressed by a formal commutative diagram.

\begin{theorem} \label{th:subgroups-automorphisms} Let $X$ be a black box group encrypting a group $G$. Assume that $G$ contains subgroups $G_1,\dots, G_l$ invariant under an automorphism $\alpha \in \mathop{{\rm Aut}}G$ and that these subgroups are encrypted in $X$ as black boxes $X_i$, $i = 1,\dots, l$, supplied with morphisms $\phi_i: X_i \longrightarrow X_i$ which encrypt restrictions $\alpha\!\mid_{G_i}$ of $\alpha$ on $G_i$.

Finally, assume $\langle G_i, i=1,\dots,l\rangle=G$.

  Then we can construct, in polynomial in $l(X)$ time, a morphism $\phi: X \longrightarrow X$ which encrypts $\alpha$.
\end{theorem}

\subsection{Construction of Frobenius maps}
\label{ssec:Frobenius}

We now use Theorem~\ref{th:automorphism} to construct a Frobenius map on a black box group $X$ encrypting $\ppsl_2(q)$ with $\qpone$ and $q=p^k$ for some $k\geqslant 1$. We make sure that the Frobenius map constructed leaves invariant the specified Borel subgroup, thus giving us access to subtler structural properties of the group.

We shall use the following result from \cite{suko12A}.

\begin{theorem}[Small characteristic] \label{cons:psl2B} {\rm \cite[Theorem 1.2]{suko12A} }
Let $X$ be a black box group encrypting $\ppsl_2(q)$, where $\qpone$ and $q=p^k$ for some $k\geqslant 1$. If $p\neq 5,7$, then there is a Monte-Carlo algorithm which constructs in $X$ strings $u$, $h$, $n$   such that there exists an \emph{(abstract)} isomorphism\/ $$\Phi: X \longrightarrow \ppsl_2(q)$$ with
\[
\Phi(u) = \begin{bmatrix} 1&1\\ 0&1\end{bmatrix}, \Phi(h) =  \begin{bmatrix} t&0\\ 0&t^{-1}\end{bmatrix}, \Phi(n) = \begin{bmatrix} 0&1\\ -1&0\end{bmatrix},
\]
where $t$ is some primitive element of the field\/ $\mathbb{F}_q$. The running time of the algorithm is quadratic in $p$ and polynomial in $\log q$.

If $p=5$ or $7$, and $k$ has a small divisor $\ell$, the same result holds where the running time is polynomial in $\log q$ and quadratic in $p^\ell$.
\end{theorem}

In notation of Theorem~\ref{cons:psl2B}, Theorem~\ref{th:automorphism} immediately yields the following remarkably useful result, see its extensions and applications in our subsequent papers \cite{suko12B,suko12C}.

\begin{theorem}[Small characteristic] \label{th:Frobenius} {\rm (Informal formulation)} Let $X$ be as in Theorem~{\rm \ref{cons:psl2B}}. Then there is a Monte-Carlo algorithm which constructs a map
 \begin{diagram}
 X & \rOnto^\phi & X
 \end{diagram}
that corresponds to the Frobenius automorphism $a \mapsto a^p$ of the field $\GF_q$ and leaves invariant subgroups $U$ and\/ $T$ and the elements\/ $u$ and\/ $w$ of\/ $X$.

 The running time of the algorithm is quadratic in $p$ and polynomial in   $\log q$.
\end{theorem}

\begin{proof} It suffices to observe that the action of the canonical Frobenius map
\[
F: \begin{bmatrix} a_{11}& a_{12}\\ a_{21}&a_{22}\end{bmatrix} \mapsto \begin{bmatrix} a_{11}^p& a_{12}^p\\ a_{21}^p&a_{22}^p\end{bmatrix}
\]
on the preimages of $\bar{u}, \bar{w}, \bar{h}$ in $\psl_2(q)$ and their images under the powers of the Frobenius map looks like that:
\bea
\begin{bmatrix} 1&1\\ 0&1\end{bmatrix}^{F^i} &=& \begin{bmatrix} 1&1\\ 0&1\end{bmatrix},\\ \begin{bmatrix} 0&1\\ -1&0\end{bmatrix}^{F^i} &=& \begin{bmatrix} 0&1\\ -1&0\end{bmatrix},\\ \begin{bmatrix} t^{p^l}&0\\ 0&t^{p^{-l}}\end{bmatrix}^{F^i} &=& \begin{bmatrix} t^{p^{l+i \pmod k}}&0\\ 0&t^{p^{-l-i \pmod k}}\end{bmatrix}\\ &=& \begin{bmatrix} t^{p^l}&0\\ 0&t^{p^{-1}}\end{bmatrix}^{p^i}.
\eea
Therefore the black box group $Y$ is generated in the direct product $X^k$ by elements
\bea
\bar{u} &=& (u,u,\dots,u)\\
\bar{w} &=& (w,w,\dots, w)\\
\bar{h} &=& (h,h^p,\dots, h^{p^{k-1}})
\eea
fits precisely in the construction described in Theorem~\ref{th:automorphism}.

It remains to notice that, by nature of its construction, the map $\alpha$ in Theorem~\ref{th:automorphism} leaves invariant elements $\bar{u}$, $\bar{w}$ and the torus $\bar T$ generated by $\bar{h}$ and hence leaves invariant the unipotent group $\bar{U} = \langle \bar{u}^{\bar{T}}\rangle $ and the Borel subgroup $\bar{U}\bar{T}$ of $Y$.
\end{proof}

Actually we have a more general construction of Frobenius maps on all untwisted Chevalley groups over finite field of odd characteristic; unlike Theorem~\ref{th:Frobenius}, it does not use unipotent elements.

\begin{theorem}[Known characteristic]
Let $X$ be a black box group encrypting a simple Lie type group $G=G(q)$ of untwisted type over a field of order $q=p^k$ for $p$ odd {\rm(}and known{\rm )} and $k >1$. Then we can construct, in time polynomial in $\log |G|$,
\bi
\item a black box $Y$ encrypting $G$,
\item a morphism $X \longleftarrow Y$, and
\item a morphism $\phi:Y \longrightarrow Y$
which encrypts a Frobenius automorphism of $G$ induced by the map $x \mapsto x^p$ on the field $\Ff_q$.
\ei
\end{theorem}

\begin{proof}  The proof is based on two applications of Theorem~\ref{th:subgroups-automorphisms}. First we consider the case when $X$ encrypts $\psl_2(\Ff_q)$. Using the standard technique for dealing with involution centralizers, we  can find in $X$ a $4$-subgroup $V$; let $E$ be the subgroup in $G=\psl_2(q)$ encrypted by $V$. Since all $4$-subgroups in $\psl_2(\Ff_q)$ are conjugate to a subgroup in $\psl_2(\Ff_p)$, we can assume without loss of generality that $E$ belongs to a subfield subgroup $H= \psl_2(\Ff_p)$ of $G$ and therefore $E$ is fixed by a Frobenius map $F$ on $G$. Now let $e_1$ and $e_2$ be two involutions in $E$, and $C_1$ and $C_2$ maximal cyclic subgroups in their centralizers in $G$; notice that $C_1$ and $C_2$ are conjugate by an element from $H$ and are $F$-invariant.

It follows from the basic Galois cohomology considerations that  $F$ acts on $C_1$ and $C_2$ as power maps $\alpha_i : c \mapsto c^{\epsilon p}$ for $p \equiv \epsilon \bmod 4$. If now we take images $X_i$ of groups $C_i$, we see that the morphisms $\phi_i: x \mapsto  x^{\epsilon p}$ of $X_i$ encrypt restrictions of $F$ to $C_i$. Obviously, $X_1$ and $X_2$ generate a black box $Y \longrightarrow X$, and we can use Theorem~\ref{th:subgroups-automorphisms} to amalgamate $\phi_1$ and $\phi_2$ into a morphism $\phi$ which encrypts $F$.

As usual, for groups $\sl_2(q)$ the same result can be achieved by essentially the same arguments as for $\psl_2(q)$. Moving to other untwisted Chevalley groups, we apply amalgamation to (encryptions of) restrictions of a Frobenius map on $G$ to (encryptions in $X$) of a family of root $\ppsl_2$-subgroups $K_i$ in $G$ forming a Curtis-Tits system in $G$ (and therefore generating $G$). Black boxes for Curtis-Tits system in classical groups of odd characteristic are constructed in \cite{suko03}, in exceptional groups in \cite{suko05}. This completes the proof. \end{proof}

\section{Oracles and revelations}
\label{sec:oracle}

In this section, we revise the classification of black box group problems and briefly discuss the role of ``oracles''.

\subsection{Monte-Carlo and Las Vegas}

This is a brief reminder of two canonical concepts for the benefit of those readers who came from the pure group theory rather than computational group theory background.

A Monte-Carlo algorithm is a randomized algorithm which gives a correct output to a decision problem with probability strictly bigger than $1/2$. The probability of having incorrect output can be made arbitrarily small by running the algorithm sufficiently many times. A Monte-Carlo algorithm with outputs ``yes'' and ``no'' is called one-sided if the output ``yes'' is always correct. A special subclass of Monte-Carlo algorithm is a Las Vegas algorithm which either outputs a correct answer or reports failure (the latter with probability less than $1/2$). The probability of having a report of failure is prescribed by the user. A detailed comparison of Monte-Carlo and Las Vegas algorithms, both from practical and theoretical point, can be found in Babai's paper \cite{MR1444127}.

\subsection{Constructive recognition}

We shall outline an hierarchy of typical black box group problems.

\begin{description}
\item[\textbf{Verification Problem}] Is the unknown group encrypted by a black box group $X$ isomorphic to the given group $G$ (``target group'')?

\item[\textbf{Recognition Problem}] Determine the isomorphism class of the group encrypted by $X$.
\end{description}

The Verification  Problem arises as a sub-problem within more complicated Recognition Problems. The two problems have dramatically different complexity. For example, the celebrated Miller-Rabin algorithm \cite{rabin80.128} for testing primality of the given odd natural number $n$ in nothing else but a  black box algorithm for solving the verification problem for the multiplicative group   $\mathbb{Z}/n\mathbb{Z}^*$ of residues modulo $n$ (given by a simple black box: take your favorite random numbers generator and generate random integers between $1$ and $n$) and the cyclic group $\mathbb{Z}/(n-1)\mathbb{Z}$ of order $n-1$ as the target group. On the other hand, if $n=pq$ is the product of primes $p$ and $q$, the recognition problem for the same black box group means finding the direct product decomposition
\[
\mathbb{Z}/n\mathbb{Z}^* \cong \mathbb{Z}/(p-1)\mathbb{Z} \oplus \mathbb{Z}/(q-1)\mathbb{Z}
\]
which is equivalent to factorization of $n$ into product of primes.

The next step after finding the isomorphism type of the black box group $X$ is

\begin{description}
\item[Constructive Recognition] Suppose that a black box group $X$ encrypts a concrete and explicitly given group $G$. Rewording a definition given in \cite{brooksbank08.885},
    \begin{quote}
        \emph{ The goal of a constructive recognition algorithm is
to construct an effective isomorphism $\Psi: G \longrightarrow X$. That is, given $g\in G$, there is an efficient procedure to construct a string $\Psi(g)$ encrypting $g$  in $X$ and given a string $x$  produced by $X$, there is an efficient procedure to construct the element $\Psi^{-1}(x) \in G$ encrypted by $X$.}
    \end{quote}
\end{description}

However, there are still no really efficient constructive recognition algorithms for black box groups $X$ of (known) Lie type over a finite field of large order $q=p^k$. The first computational obstacles for known algorithms
\cite{brooksbank01.79,brooksbank03.162,brooksbank08.885,brooksbank01.95,brooksbank06.256,celler98.11,conder06.1203,leedham-green09.833} are the need to construct unipotent elements in black box groups, \cite{brooksbank01.79,brooksbank03.162,brooksbank08.885,brooksbank06.256,brooksbank01.95,celler98.11} or to solve discrete logarithm problem for matrix groups \cite{conder01.113,conder06.1203,leedham-green09.833}.

Unfortunately,
the proportion of the unipotent elements in $X$ is $O(1/q)$ \cite{guralnick01.169}. Moreover the probability that the order of a random element is divisible by $p$ is also $O(1/q)$, so one has to make  $O(q)$ (that is, \emph{exponentially many}, in terms of the input length $O(\log q)$ of the black boxes and the algorithms) random selections of elements in a given group to construct a unipotent element. However, this brute force approach is still working for small values of $q$, and  Kantor and Seress \cite{kantor01.168} used it to develop an algorithm for recognition of black box classical groups. Later the algorithms of \cite{kantor01.168} were upgraded to polynomial time constructive recognition algorithms \cite{brooksbank03.162, brooksbank08.885, brooksbank01.95, brooksbank06.256} by assuming the availability of additional \emph{oracles}:
 \begin{itemize}
 \item the \emph{discrete logarithm oracle} in $\mathbb{F}_q^*$, and
  \item the  \emph{$\sl_2(q)$-oracle}.
   \end{itemize}
Here, the \emph{$\sl_2(q)$-oracle} is a procedure for constructive recognition of $\sl_2(q)$; see  discussion in \cite[Section~3]{brooksbank08.885}.

\begin{quote}
\textbf{\emph{We emphasize that in this and subsequent papers we are using neither the discrete logarithm oracle in $\mathbb{F}_q^*$ nor the  $\sl_2(q)$-oracle.}}
\end{quote}

\subsection{On oracles and revelations: an example from even characteristic}

We have to admit that the concept of constructive recognition modulo the use of unrealistically powerful oracles makes us uncomfortable. We feel that the use of excessively powerful and blunt tools leads to loss of essential (and frequently very beautiful) theoretical details. Instead, we propose to use all ``\emph{fifty shades of black}'' and exploit all available gradations of black (that is, a subtler hierarchy of complexity of black box problems)  in  development of practically useful algorithms.  Our papers \cite{suko12A,suko12B,suko12C} provide a number of concrete examples where this alternative approach has happened to be fruitful.

In the present paper, we wish to dispel some mystic of the $\sl_2(q)$-oracle by analyzing the structure of the black box group $X$ encrypting $\sl_2(2^n)$ using formally a more modest assumption: that we are given an involution $r \in X$. We shall say that $r$ is obtained \emph{by revelation}, to acknowledge that this assumption is quite unnatural in practical applications.

Still, we feel that there is a difference between a revelation or epiphany (which, by their nature, are non-reproducible, unique events) and an appeal to an oracle; indeed, there is an implicit assumption that the oracle can be approached for advice again and again.

\begin{theorem}[Small characteristic]
Let $X$ be a black box group  encrypting $\sl_2(2^n)$ for some \textup{(}perhaps unknown\textup{)} $n$. We assume that we  are given an involution $u \in X$.

Then there exists a Monte-Carlo algorithm which constructs, in  polynomial in $l(X)$ time,
 \bi
 \item a black box field\/ $\uu$ encrypting\/ $\Ff_{2^n}$,  and
 \item  a  polynomial in $l(X)$ time isomorphism
\[\Phi: \sl_2(\uu) \longrightarrow X.\]
\ei
\label{th:sl(2,2n)}
\end{theorem}

\subsection{Structure recovery}
\label{ssec:structure-recovery}

Theorem~\ref{th:sl(2,2n)} is an example of a class of results which we call \emph{structure recovery theorems}.\footnote{We extend our definition from \cite{suko12A} where it refers to a special case of the present one.}

Suppose that a black box group $X$ encrypts a concrete and explicitly given group $G=G(\Ff_q)$ of Chevalley type $G$ over a explicitly given finite field $\Ff_q$. To achieve \emph{structure recovery} in $X$ means to construct, in probabilistic polynomial time in $\log |G|$,
    \bi
     \item a black box field $\KK$ encrypting $\Ff_q$, and
     \item a probabilistic polynomial time morphism $$\Psi: G(\KK) \longrightarrow X.$$
    \ei

This new concept requires a detailed discussion.

Recall that simple algebraic groups (in particular, Chevalley groups over finite fields) are understood in the theory of algebraic groups as functors from the category of unital commutative rings into the category of groups; most structural properties of a Chevalley group are encoded in the functor; the field mostly provides the flesh on the bones. Remarkably, this separation of flesh from the bones is  very prominent in the black box group theory. Here, we wish to mention a few from many constructions from our subsequent paper  \cite{suko12B}  which illustrate this point.

\begin{theorem}[Known characteristic] \label{th:sun-in-sln} \textbf{\cite{suko12B}}
Let $X$ be a black box group encrypting the group $\sl_n(q^2)$ for $q$ odd, $q=p^k$ for some $k$ (perhaps unknown) and a known prime number $p$. Then we can construct, in time polynomial in $\log q$ and $n$, a black box group $Y$   encrypting the group $\su_n(q)$ and a morphism $Y \hookrightarrow X$. If in addition $n$ is even and $n =2m$, we can do the same with a black box group $Z$ encrypting $\sp_{2m}(q)$ and a morphism $Z \hookrightarrow Y$.
\end{theorem}

An important feature of the proofs of this and other similar results in \cite{suko12B} is that they never refer to the ground fields of groups and do not involved any computations with unipotent elements. In fact, we interpret morphisms between functors
\[
\sp_{2m}(\cdot)  \hookrightarrow \su_n(\cdot) \hookrightarrow \sl_n(\cdot^2).
\]
within  our black boxes.

This example shows that a modicum of categorical language is useful for the theory as well as for its implementation in the code since it suggests a natural structural approach to development of the computer code.

Another example of a ``category-theoretical'' approach is provided by a very elementary, but also very important observation that the graph of a group homomorphism $G \longrightarrow H$  is a subgroup of $G \times H$. Therefore it is natural to identify a morphism $\mu:X \longrightarrow Y$ of black box groups with its graph $M < X \times Y$. In its turn, the black box $M$ is a morphism $M \longrightarrow X\times Y$. In practice this could  mean (although in some cases a more sophisticated construction is used) that we take some strings $x_1,\dots,x_k$ generating $X$ and their images $y_1=\mu(x_1),\dots, y_k=\mu(x_k)$ in $Y$ and use the product replacement algorithm to run a black box for the subgroup \[M = \langle (x_1,y_1),\dots,(x_k,y_k)\rangle \leqslant X \times Y\] which is of course exactly the graph $\{\, (x,\mu(x))\,\}$ of the homomorphism $\mu$. Random sampling of the black box $M$ returns strings $x\in X$ with their images $\mu(x)\in Y$ already attached. This doubles the computational cost of the black box for $X$, but allows us to do constructions like the following one.

\begin{theorem}[Known characteristic] \textbf{\cite{suko12B}}
\label{th:G2-in-SL8}
Let $X$ be a black box group encrypting the group $\sl_8(F)$ for a field $F$ of (unknown) odd order $q = p^k$ but known $p= {\rm char}\,F$.  Then we can construct, in time polynomial in $\log |F|$, a chain of black box groups and morphisms
\[
U \hookrightarrow V \hookrightarrow W \hookrightarrow X
\]
that encrypts the chain of canonical embeddings
\[
\mathop{G_2}(F) \hookrightarrow \so_7(F) \hookrightarrow\so_8^+(F) \hookrightarrow \sl_8(F).
\]
\end{theorem}

Again, these our constructions  (and even the embedding $^3{\rm D}_4(q) \hookrightarrow  \so_8^+(q^3)$,  also done in \cite{suko12B}) are ``field-free'' and, moreover, ``characteristic-free''.

Another aspect of the concept of ``structure recovery'' is that it follows an important technique from the model-theoretic algebra: interpretability of one algebraic structure in another, see, for example, \cite{borovik94-BN}. Construction of a black box field in a black box group in Theorems~\ref{th:sl(2,2n)} and \ref{cons:psl2} closely follows this model-theoretic paradigm.

\subsection{Black box fields}
\label{ssec:bbf}

We define black box fields by analogy with black box groups, the reader may wish to compare the  exposition in this section with \cite{boneh96.283}.

A \emph{black box} (finite) \emph{field} $\KK$ is an oracle or an algorithm operating on $0$-$1$ strings of uniform length (input length) which encrypts a field of known characteristic $p$. The oracle can compute $x+y$, $xy$ and compares whether $x=y$ for any strings $x,y \in \kk$.  We refer the reader to \cite{boneh96.283, maurer07.427} for more details of black box fields and their applications to cryptography.

In this paper, we shall be using some results about the isomorphism problem of black box fields \cite{maurer07.427}, that is, the problem of constructing an isomorphism and its inverse between $\kk$ and an explicitly given finite field $\Ff_{p^n}$. The explicit data for a finite field of cardinality $p^n$ is defined to be a system of {\emph{structure constants} over the prime field, that is} $n^3$ elements $(c_{ijk})_{i,j,k=1}^n$ of the prime field {$\mathbb{F}_p = \mathbb{Z}/p\mathbb{Z}$ (represented as integers in $[0,p-1]$)} so that $\mathbb{F}_{p^n}$ becomes a field with ordinary addition and multiplication by elements of $\mathbb{F}_p$ and multiplication is determined by
\[
s_i s_j =\sum_{k=1}^n c_{ijk}s_k,
\]
where $s_1, s_2, \dots, s_n$ denotes a basis of $\mathbb{F}_{p^n}$ over $\mathbb{F}_p$. The concept of explicitly given field of order $p^n$ is robust; indeed, Lenstra Jr.\  has shown in \cite[Theorem 1.2]{lenstra91.329} that for any two fields $A$ and $B$ of order $p^n$ given by two sets of structure constants $(a_{ijk})_{i,j,k=1}^n$ and $(b_{ijk})_{i,j,k=1}^n$ an isomorphism $A \longrightarrow B$ can be constructed in polynomial in $n\log p$ time.

Maurer and Raub  \cite{maurer07.427} proved  that the isomorphism problem for a black box field $\kk$ and an explicitly given field $\Ff_{p^n}$ is reducible in polynomial time to the same problem for the prime subfield in $\kk$ and $\Ff_p$. Hence, for small primes $p$, one can construct an isomorphism between $\KK$ and $\Ff_{p^n}$ in time polynomial in $n\log p$ and linear in $p$.

In our construction of a black box field, we use the so called \textit{primitive prime divisor} elements in the field of size $p^n$. A prime number $r$ is said to be a primitive prime divisor of $p^n-1$  if $r$ divides $p^n-1$ but not $p^i-1$ for $1\leqslant i<n$. By \cite{zsigmondy92.265}, there exists a primitive prime divisor of $p^n-1$ except when $(p,n)=(2,6)$, or $n=2$ and $p$ is a Mersenne prime. Here, we shall note that the Mersenne primes which are less than 1000 are 3, 7, 31, 127. We call a group element a $ppd(n,p)$-element if its order is odd and divisible by a primitive prime divisor of $p^n-1$.

\section{Application of Frobenius maps: structure recovery of $\ppsl_2(q)$, $\qpone$}\label{sec:frobenius}

We remind that in all theorems and conjectures stated in this paper, we assume that black boxes for groups satisfy Axioms BB1--BB4; in particular, they come with known and computationally feasible global exponent (Axiom BB4).

For the structure recovery of $\ppsl_2(q)$, we need to recall the Steinberg generators of $\ppsl_2(q)$ as introduced by Steinberg \cite[Theorem 8]{steinberg1968}. We use  notation from \cite{suko12A}.

Let $G=\sl_2(q)$. Then set the Steinberg generators of $G$ as
\begin{eqnarray*}
\mathbf{u}(t)= \left[
\begin{array}{cc}
1  & t  \\
0   &  1 \\
\end{array}
\right],
\,
\mathbf{v}(t)= \left[
\begin{array}{cc}
1  & 0  \\
t   &  1 \\
\end{array}
\right],
\,
\mathbf{h}(t)= \left[
\begin{array}{cc}
t  & 0  \\
0   &  t^{-1} \\
\end{array}
\right],
\,
\mathbf{n}(t)= \left[
\begin{array}{cc}
0  & t  \\
-t^{-1}   &  0 \\
\end{array}
\right]
\end{eqnarray*}
where $t \in \mathbb{F}_q$ and in addition $t\neq 0$ in $\mathbf{h}(t)$ and $\mathbf{n}(t)$.

The group $\psl_2(q)$ is obtained from  $\sl_2(q)$ by factorizing over the relation
\[
\mathbf{h}(t) = \mathbf{h}(-t).
\]
Abusing notation, we are using for elements in $\psl_2(q)$ the same matrix notation as for their pre-images in $\sl_2(q)$.

 It is straightforward to check that
\begin{eqnarray}\label{opp:uni}
\mathbf{u}(t)^{\mathbf{n}(s)}=\mathbf{v}(-s^{-2}t), \,  \mathbf{u}(1)^{\mathbf{h}(t)}=\mathbf{u}(t^{-2}) \mbox{ and } \mathbf{n}(1)^{\mathbf{h}(t)}=\mathbf{n}(t^{-2}).
\end{eqnarray}
Moreover,
\begin{eqnarray}\label{tori:weyl}
\mathbf{n}(t)=\mathbf{u}(t)\mathbf{v}(-t^{-1})\mathbf{u}(t) \mbox{ and } \mathbf{h}(t)=\mathbf{n}(t)\mathbf{n}(-1).
\end{eqnarray}
It is well-known that  \[G=\langle \mathbf{u}(t), \mathbf{v}(t) \mid t\in \Ff_q \rangle,\] see, for example, \cite[Lemma 6.1.1]{carter1972}. Therefore, by (\ref{opp:uni}) and (\ref{tori:weyl}), \[G= \langle \mathbf{u}(1),\mathbf{h}(t),\mathbf{n}(1)\mid t \in \Ff_q^*\rangle;\]
notice that actually $G$ is generated by three elements
\[G= \langle \mathbf{u}(1),\mathbf{h}(t),\mathbf{n}(1)\rangle\]
where we can take $t$ as an arbitrary $ppd(k,p)$-element of the field\/ $\mathbb{F}_{p^k}$.

In this section, we prove the following theorem.

\begin{theorem}[Small characteristic]\label{cons:psl2}
Let $X$ be a black box group encrypting the group $G\cong \ppsl_2(q)$, where $\qpone$ and $q=p^k$ for some $k\geqslant 1$ {\rm(}perhaps unknown{\rm)}. If $p\neq 5,7$, then there is a Monte-Carlo algorithm which constructs,  in time quadratic in $p$ and polynomial in $\log q$,
\bi
\item  a black box field\/ $\KK$ encrypting\/ $\Ff_{q}$,  and
\item a quadratic in $p$ and polynomial in $\log q$ time isomorphism
\[
\Phi: (P)SL_2(\KK) \longrightarrow X.
\]
\ei
If $p=5$ or $7$, and $k$ has a small divisor $\ell$, the same result holds where the running time is polynomial in $\log q$ and quadratic in $p^\ell$.
\end{theorem}

Theorem~\ref{cons:psl2} is used in our  paper \cite{suko12E} as the basis of recursion in the proof of the following structure recovery theorem for classical groups in small characteristics.

\begin{theorem}[Small characteristic] \textbf{{\rm \cite{suko12E}}} \label{cons:classical}
Let $X$ be a black box  group encrypting one of the classical groups ${\rm G}(q) \simeq \ppsl_{n+1}(q)$, $\ppsp_{2n}(q)$, $\om_{2n+1}(q)$ or\/ $\ppom_{2n}^+(q)$, where $\qpone$ and $q=p^k$ for some $k\geqslant 1$ {\rm(}$k$ and the type of the group are perhaps unknown{\rm)}.

If $p\neq 5,7$, then there is a Monte-Carlo algorithm which constructs,  in time quadratic in $p$ and polynomial in $\log q$,
\bi
\item  a black box field\/ $\KK$ encrypting\/ $\Ff_{q}$,  and
\item a quadratic in $p$ and polynomial in $\log q$ time isomorphism
\[
\Phi: {\rm G}(\KK) \longrightarrow X.
\]
\ei
If $p=5$ or $7$, and $k$ has a small divisor $\ell$, the same result holds where the running time is polynomial in $\log q$ and quadratic in $p^\ell$.
\end{theorem}

\label{sec:proofpsl2odd}

\subsection{Proof of Theorem~\ref{cons:psl2}, general case}
\label{ssec:pls2}

Our aim is to present an  algorithm which produces a black box field $\KK$ and an isomorphism
\[
\varphi: \sl_2(\KK) \rightarrow X.
\]

\ben

\item We use Theorem~\ref{cons:psl2B} as applied to our black box group $X$, so  $u$, $h$, $n$   are string in $X$ such that
\[
\Phi(u) = \begin{bmatrix} 1&1\\ 0&1\end{bmatrix}, \Phi(h) =  \begin{bmatrix} t&0\\ 0&t^{-1}\end{bmatrix}, \Phi(n) = \begin{bmatrix} 0&1\\ -1&0\end{bmatrix}
\]
for some abstract isomorphism\/ $$\Phi: X \longrightarrow \ppsl_2(q);$$ here, $t$ is some primitive element in $\mathbb{F}_{q}$, $q=p^k$. We shall note that we only know the existence of the map $\Phi$. Let $\tilde{h}$ be a $ppd(k,p)$-element produced by taking some power of $h$ and $$\Phi(\tilde{h}) =  \begin{bmatrix} \tilde{t}&0\\ 0&\tilde{t}^{-1}\end{bmatrix}.$$

\item   We consider the cyclic subgroup $T = \langle \tilde{h} \rangle$ and the unipotent subgroup $U = \langle u^T \rangle$ in $X$. Observe that $U$ is the full unipotent subgroup of $X$ since the order of $\tilde{h}$ is a $ppd(k,p)$-element in $\mathbb{F}_{p^k}$.

\item Now we start introducing on $U$ a structure of field $\KK$ isomorphic to $\GF_q$. First, for any $u_1, u_2 \in U$, we define an addition on $\KK$ by setting
\[
u_1 \oplus u_2 = u_1u_2.
\]
For the multiplication on $\KK$, we set the element $u$ as the unity  of $\KK$. Since $\tilde {h}$ is a $ppd(k,p)$-element, it has odd order $m$ and the element $\sqrt{\tilde{h}}:= \tilde{h}^{(m+1)/2}$ has the property that $\sqrt{\tilde{h}}^2=\tilde{h}$. We also set
\[
s:=u^{\sqrt{\tilde{h}}}.
\]
Notice that
\[
\left[
\begin{array}{cc}
1  & 1  \\
0   &  1 \\
\end{array}
\right]^{
\,
\left[
\begin{array}{cc}
\sqrt{\tilde{t}}  & 0  \\
0   &  \sqrt{\tilde{t}^{-1}} \\
\end{array}
\right]
} = \left[\begin{array}{cc}
1  & \tilde{t}^{-1}  \\
0   &  1 \\
\end{array}
\right].
\]
Hence $s$ can be seen as an element in $\KK$ corresponding to $\tilde{t}^{-1}$, 
and after setting $s^{i} = u^{(\sqrt{\tilde{h}})^{i}}$, the elements
\[
s, s^{2}, \dots, s^{{k-1}}, s^{k}
\]
form a polynomial basis of $\KK$ over the prime field $\LL \simeq\GF_p$.  The additive groups of $\LL$ is cyclic of order $p$. We have already fixed the identity element $1$ of $\KK$ and hence of $\LL$, which uniquely defines the multiplicative structure on $\LL$.

For  $w\in U$, we define the product
\[
w \otimes s^l = w^{h^l}
\]
 and expanded by linearity to product of any two elements in $\KK$. We still do not know, however, why this operation can be carried out in feasible time---but we should be reassured that at least the product $w\otimes s^l$ can be computed in time polynomial in $\log q$. So at this stage we treat $\KK$ a \emph{partially} polynomial time black box field: random generation, comparison,  and addition of elements in $\KK$ can be carried out in polynomial in $\log q$, as well as multiplication of an arbitrary element in $\KK$ by some specific elements.

\item In view of Theorem~\ref{th:Frobenius}, we have the Frobenius map $\phi$ on our black box group $X \simeq \ppsl_2(q)$ which leaves $U$ and $T$ invariant and induces the Frobenius map $F$ on $U$. This allows us to introduce on $U$ the Frobenius trace $\Tr: U \to \GF_p$
\[
\Tr(x) = x \oplus x^F \oplus x^{F^2}\oplus \cdots\oplus x^{F^{k-1}}
\]
and the trace form, that is, the non-degenerate symmetric $\mathbb{F}_p$-bilinear form given by
\[
\langle x,y\rangle = \Tr(x\otimes y).
\]

It is interesting that the Frobenius map and the trace form of our future black box field are introduced \emph{before} the field multiplication!

We do not know yet whether the evaluation of the trace form on $\KK$ is computationally feasible, but we can compute in polynomial in $\log q$ time the values $w\otimes s^l = w^{h^l}$  and of $\langle w, s^l\rangle$ for arbitrary $w \in \KK$ and powers of $s$. In particular, this allows us to compute the matrix of the trace form
\[
A = (a_{ij})_{k\times k}, \quad a_{i,j} = \langle s^{{i}},s^{{j}}\rangle, \quad i,j, = 1,2,\dots, k.
\]

\item We are now in position to introduce in $\KK$ an explicit structure of a $\LL$ vector space by computing the decomposition of an arbitrary element $w\in \KK$ with respect to the basis $s, s^2,\dots, s^{{k}}$. Indeed, for an arbitrary element $w\in \KK$, set
    \[
    w = \alpha_1s\oplus \alpha_2s^2\oplus \cdots\oplus  \alpha_{k}s^{{k}}
    \]
    and
    \[
    \beta_j = \langle w, s^{j}\rangle, \quad j = 1, 2, \dots, k.
    \]
The coefficients $\beta_j$ are computable in time polynomial in $\log p$ and $k$:
    \[
    \beta_{j} = \langle w, s^{j}\rangle = \sum_{i=1}^{k} \alpha_i a_{ij},
    \]
    which in matrix notation becomes
    \[
    (\beta_1,\dots,\beta_{k}) =  (\alpha_1,\dots, \alpha_{k})\cdot A,
    \]
    and therefore
    \[
    (\alpha_0,\dots, \alpha_{k-1}) = (\beta_0,\dots, \beta_{k-1})\cdot A^{-1}.
    \]

\item   We can now decompose products $s^{i}\otimes s^{j}$ with respect to the basis $s, s^2,\dots, s^{{k}}$ and thus find the structure constants $c_{ijl}$ for this basis:
\[
s^{i}\otimes s^{j} = \sum_{l=1}^{k}  c_{ijl}s^{l}.
\]
Of course now we are in position to multiply any two elements in $\KK$, and, as we can easily see,  in time polynomial in  $\log q$. Now, we shall use the algorithms in \cite{allombert02.332,lenstra91.329, maurer07.427} to construct the isomorphism between $\mathbb{F}_{p^k}$ and $\mathbb{K}$, see discussion in Section~\ref{ssec:bbf}.

\item Now, we construct $\ppsl_2(\KK)$ by using the Steinberg generators, see Section \ref{sec:frobenius}.  Recall that the element $s \in \KK$ corresponds to the element $\tilde{t}^{-1} \in \mathbb{F}_q$ where $\tilde{t}$ is a $ppd(k,p)$-element in $\mathbb{F}_q$, so $s$ is a  $ppd(k,p)$-element in $\KK$.  We construct, in $\ppsl_2(\KK)$, the elements encrypting the strings
\[
\mathbf{u}(1), \mathbf{h}(s^{-1}), \mathbf{n}(1)
\]
by using the isomorphism between the fields $\mathbb{F}_{p^k}$ and $\KK$ constructed in Step 6.

\item Our first assignments are $\mathbf{u}(1) \mapsto u$ and $\mathbf{h}(s^{-1}) \mapsto \tilde{h}$. Now we need to construct the element in $X$ encrypting the string $\mathbf{n}(1)$. Note that the element $n\in X$, which was constructed in Step 1,  need not necessarily be the element corresponding to $\mathbf{n}(1)$. Therefore we shall replace the original element $n$ by the one that corresponds to $\mathbf{n}(1)$. Recall that the elements $u,n \in X$ are indeed computed inside a subgroup isomorphic to $\ppsl_2(p)$ or $\psl_2(p^2)$ depending on $\ppone$ or $\pmone$, respectively \cite{suko12A}. For simplicity, we may assume that this subgroup encrypts $\ppsl_2(p)$ and the following computations are carried out in this black box subgroup. Note that raising the element $h$ to the power so that the resulting element $h_0$ has order $(p-1)/2$ and belongs to this subgroup isomorphic to $\ppsl_2(p)$.

We compute all $v:=(u^{-1})^{h_0^kn}$ for $k=1, \ldots p-1$, and check which of the elements of the form
\[
uv^{-1}u
\]
has order $4$ (Recall that, by (\ref{opp:uni}) and (\ref{tori:weyl}), we have $\mathbf{u}(1)^{\mathbf{n}(s)}=\mathbf{v}(-s^{-2})$ and $\mathbf{n}(t)=\mathbf{u}(t)\mathbf{v}(-t^{-1})\mathbf{u}(t)$). Observe that there are only two elements of the form $uv^{-1}u$ of order 4 and they  correspond to the elements $\mathbf{n}(1)$ and $\mathbf{n}(-1)$. Now we need to distinguish $\mathbf{n}(1)$ from $\mathbf{n}(-1)$. Recall also that, by (\ref{opp:uni}) and (\ref{tori:weyl}), we have
$$\mathbf{n}(1)^{\mathbf{h}(t)}=\mathbf{n}(t^{-2}), \mathbf{u}(1)^{\mathbf{h}(t)}=\mathbf{u}(t^{-2}), \mathbf{v}(1)^{\mathbf{h}(t)}=\mathbf{v}(t^2)$$ and
\begin{equation}\label{nn1}
\mathbf{n}(t^{-2})=\mathbf{u}(t^{-2})\mathbf{v}(-t^2)\mathbf{u}(t^{-2}).
\end{equation}
 Now it is easy to see that if one of the elements of the form $uv^{-1}u$ of order 4 corresponds to the Weyl group element $\mathbf{n}(-1)$, then Equation (\ref{nn1}) is not satisfied. Hence the Weyl group element $h_0^kn$ which produces the element $uv^{-1}u$ satisfying Equation (\ref{nn1}) is the desired Weyl group element, say $\tilde{n}$.

\item Observe that the following map
\begin{eqnarray*}
\ppsl_2(\KK) & \longrightarrow & Y\\
\mathbf{u}(1) &\mapsto & u\\
\mathbf{h}(s^{-1}) &\mapsto & \tilde{h}\\
\mathbf{n}(1) &\mapsto&\tilde{n}
\end{eqnarray*}
is an isomorphism.

\end{enumerate}

Notice that the algorithm described above provides a proof of Theorem \ref{cons:psl2}.

\subsection{A more straightforward treatment of  $\sl_2(p)$}

Because of its importance, we give a streamlined construction of an isomorphism between $\sl_2(p)$, $p \equiv 1 \bmod 4$, and a black box group $X$ encrypting $\sl_2(p)$. Notice, in this case, that we may assume that the field structure of $\Ff_p$ is available. Hence, we shall construct the elements in $X$ encrypting the images of $\mathbf{u}(1), \mathbf{h}(t)$ and $\mathbf{n}(1)$ where $0,1 \neq t\in \Ff_p$ in $X$.

Step 1: Using Theorem~\ref{cons:psl2B}, we select in $X$ a unipotent element $u$, a toral element $h$  normalizing the root subgroup containing $u$, and $n$ a Weyl group element for the torus containing $h$. Our fist assignment is $ \mathbf{u}(1) \mapsto u$.

Step 2: Recall that for a given $\mathbf{h}(t)$ we have $\mathbf{u}(1)^{\mathbf{h}(t)}=\mathbf{u}(t^{-2})$. Assume that $\mathbf{u}(t^{-2})=\mathbf{u}(k)=\mathbf{u}(1)^k$ for some $k\in \{1,2,\ldots, p-1\}$.

Now we check whether $u^h=u^k$ in $X$. If not, then some power $\ell$ of $h$ has this property, that is, $u^{h^\ell}=u^k$. Observe that $\ell$ is necessarily relatively prime to $p-1$ so that the resulting element $h^\ell$ generates the torus. We replace $h$ with $h^\ell$ and assign $\mathbf{h}(t) \mapsto h$.

Step 3: Now we compute $\mathbf{n}(1)$ by using the same arguments in Step 9 of the algorithm in  Section~\ref{ssec:pls2}. Thus we have an isomorphism
\begin{eqnarray*}
\sl_2(p) & \longrightarrow & X\\
\mathbf{u}(1) &\mapsto & u\\
\mathbf{h}(t) &\mapsto & h\\
\mathbf{n}(1) &\mapsto&n.
\end{eqnarray*}

\section{A Revelation and Its Reverberations: Proof of Theorem~\ref{th:sl(2,2n)}}
\label{sec:proofSl2odd}

\subsection{Proof of Theorem~\ref{th:sl(2,2n)}}

We describe an algorithm which produces a black box field $\uu$ and an isomorphism
\[\Phi: \sl_2(\uu) \longrightarrow X.\]

\ben
\item We take our revelation involution $r$ and consider strongly real elements of the form $r^x\cdot r$ for random $x\in X$, and raising them to appropriate powers, find an element $\theta$ of order $3$ inverted by $r$.

\item Set     $v = \theta r$ and $w = \theta^2 r$. Observe that $v$ and $w$ are involutions and $L = \langle \theta\rangle \langle r \rangle$ is the dihedral group of order $6$.

\item Observe that all dihedral subgroups of order $6$ in $X$ are conjugate in $X$ and therefore we can assume without loss of generality that $L \cong \sl_2(2)$ encrypts a subfield subgroup of $\sl_2(2^n)$. In particular, there exist a system of Steinberg generators of $\sl_2(2^n)$,
    \begin{eqnarray*}
\mathbf{u}(t)= \left[
\begin{array}{cc}
1  & t  \\
0   &  1 \\
\end{array}
\right],
\,
\mathbf{v}(t)= \left[
\begin{array}{cc}
1  & 0  \\
t   &  1 \\
\end{array}
\right],
\,
\mathbf{h}(t)= \left[
\begin{array}{cc}
t  & 0  \\
0   &  t^{-1} \\
\end{array}
\right],
\,
\mathbf{n}(t)= \left[
\begin{array}{cc}
0  & t  \\
t^{-1}   &  0 \\
\end{array}
\right]
\end{eqnarray*}
for $t\in \GF_{2^n}$ and $t\neq 0$ for $\mathbf{h}(t)$ and $\mathbf{n}(t)$, and such that
$r$, $v$ and $n$ encrypt $\mathbf{u}(1)$, $\mathbf{v}(1)$, and $\mathbf{n}(1)$, correspondingly.

\item The standard procedure for construction of centralizers of involutions \cite{borovik02.7,bray00.241} produces unipotent subgroups $U = C_X(r)$ and $V = C_X(v)$. If we set $$H= \langle h(t) \mid t \in \GF_{2^n}\rangle$$ (warning: this subgroup is not constructed yet) then $B^+ = UH = N_X(U)$ and $B^- = VH=N_X(V)$ are Borel subgroups in $X$.

\item Observe that if $x\in X$ is such that $u^x \in U$ for some $1\ne u \in U$ then $x \in B$.

\item We can identify action of $H$ on $U$ by conjugation with the action of $B/U$ on $U$. Observe that for any two involutions $s, t\in U$ there is a unique $\bar{b} \in B/U$ such that $s^{\bar{b}}= t$.

\item Using the double conjugation trick, we can find, for any given involutions $s$ and $t$ in $U$ an element $x$ in $X$ (and hence in $B$) such that $s^x = t$. This is done in the following way: notice that the exponent of $\sl_2(2^n)$ is $2\cdot(2^n-1)(2^n+1)$ and therefore if $y\in X$ is an element of odd order than $y^{2^{2n}-1} =1$. By conjugating $s$ by a random element $z\in X$, find an involution $r = s^z$ such that elements $y_1 = sr$ and $y_2 = rt$ have  odd order. Then it can be checked directly that
    \[
    s^{\left((sr)^{2^{2n-1}}\right)} = r \quad \mbox{ and } \quad  r^{\left((rt)^{2^{2n-1}}\right)} = t \]
    and
    \[
    x = (sr)^{2^{2n-1}} \cdot (rt)^{2^{2n-1}}
    \]
    has the desired property $s^x = t$.
    By the previous point, the coset $xU$ in $B/U$ is uniquely determined.

    (The same idea of ``local conjugation'' of involutions is used by Ballantyne and Rowley for construction of centralizers of involutions in black box groups with expensive generation of random elements \cite{Ballantyne2013}.)

\item Treating the subgroup $B$ as a black box, we have
\[
U = \{\, x \in B \mid x^2=1 \}.
\]
Therefore after introducing on $B$ a new equality relation
\[
x \equiv y \mbox{ if and only if } (xy^{-1})^2 =1
\]
we get a black box $\TT$ for the factor group $T = B/U$. Notice that there is a natural action of $\TT$ on $U$ by conjugation and that notation
\(u^t\) for $u \in U$ and $t \in \TT$ is not ambiguous.

\item Now we construct a black box field $\uu$. We start with the multiplicative group $\uu^*$ of $\uu$ which we define as the graph of the orbit action map of $\TT$ onto the orbit $r^\TT$. Namely, $\uu^*$ is the set of all pairs
 $(t, s)$ with $t \in \TT$ and $s \in U\smallsetminus \{1\}$ such that $r^t =s$. We define in $\uu$ multiplication $\otimes$ by the rule
 \[
 (t_1,u_1) \otimes (t_2,u_2) = (t_1t_s, r^{t_1t_2}).
 \]
 In particular, the element $\mathbf{1} = (1,r)$ plays the role of the identity element in $\uu^*$.

 Then we define the zero element of $\uu$ as
 \[
\mathbf{0}=(1,1),
\]
set
\[
\uu = \uu^* \cup \{0\}
\]
(and use lower case boldfaced letter to denote elements $\mathbf{u} \in \uu$), and define
\[
\mathbf{0} \otimes \mathbf{u} = \mathbf{u} \otimes \mathbf{0} \mbox{ for all } \mathbf{u} \in \uu^*.
\]
Finally, we define on $\uu$ addition $\oplus$ by setting
\bea
\mathbf{0} \oplus \mathbf{u} = \mathbf{u} \oplus \mathbf{0} &=& \mathbf{u}\\
 \mathbf{u} \oplus \mathbf{u} &=& \mathbf{0}\\
 (t_1,u_1) \oplus (t_2, u_2) &=& (t, u_1u_2)
 \eea
where in the last line $u_1 \ne u_2$ (and thus $u_1u_2 \ne 1$) and $t\in \TT$ is chosen to send $r$ to $u_1u_2$, that is, $r^t = u_1u_2$.
It follows that the inverse $\mathbf{u}^{-1}$ of $\mathbf{u} = (t,u) \ne \mathbf{0}$ with respect to multiplication $\otimes$ is equal to $(t^{-1},r^{t^{-1}})$.

\item So we have a black box field $\uu$ interpreted in the Borel subgroup $B = N_X(C_X(r)))$ of the black box group $X$ and such that $X$ encrypts $\sl_2(\uu)$.

It will be convenient to use traditional notation and denote $1 = u(\mathbf{0})$, and write, for elements $\mathbf{t}\in \uu^*$, $u = u(\mathbf{t})$ if $\mathbf{t} = (t,u)$. In particular, $r = u(\mathbf{1})$. This gives us a parametrization of $U$ by elements of the black box field $\uu$.

\item Now we transfer the black box field parametrization from $U$ to $V$ by setting $v(\mathbf{0}) =1$ and for setting for non-identity elements $v\in V$
\[
v= v(\mathbf{t}) \mbox{ if } v^w = u(\mathbf{t}).
\]

We set further
\[
n(\mathbf{t}) = u(\mathbf{t})v(\mathbf{t}^{-1})u(\mathbf{t}),
\]
so that this agrees with computation in $L\cong \sl_2(2)$, yielding
\[
n(\mathbf{1}) = w,
\]
and finally set
\[
h(\mathbf{t}) = n(\mathbf{t})n(\mathbf{1}).
\]
Notice that
\[
\left\{ h(\mathbf{t}) \mid \mathbf{t} \in \uu^*\right\} = N_X(V) \cap N_X(U)
\]
is the uniquely determined maximal torus in $X$ normalizing the both $V$ and $U$. We denote it by $H$.

\item We can now construct an isomorphism
    \[\Psi: \sl_2(\uu) \longrightarrow X.\]
First of all, recall that matrices from $\sl_2(\uu)$ are quadruples
\[
\bbm a_{11} & a_{12} \\ a_{21} & a_{22}\ebm
\]
of strings $a_{ij}$ generated by black box $\uu$, with matrix addition and multiplication defined with respect to operations $\oplus$ and $\otimes$.
    \ben
    \item Notice easy-to-check identities over any field of characteristic $2$:
    \ben
    \item given $a$, $b$, and $d$ such that $bc = 1$, we have
    \[\bbm 0 & b \\ c& d\ebm = \bbm 0 & 1 \\ 1 & 0\ebm \bbm c & 0 \\ 0 &b\ebm\bbm 1 & bd \\ 0 & 1\ebm;
    \]
    \item for $a\ne 0$ and $ad-bc =1$,
    \[
    \bbm a & b \\ c & d\ebm = \bbm a & 0 \\ 0 & a^{-1}\ebm \bbm 1 & 0 \\ ac &1\ebm\bbm 1 & a^{-1}b \\ 0 & 1\ebm.
    \]
    \een
    \item Therefore we can map
    \bea
    \Psi: \bbm \mathbf{0} & \mathbf{b} \\ \mathbf{c} & \mathbf{d}\ebm &\mapsto&  n(\mathbf{1})h(\mathbf{c})u(\mathbf{b}\otimes\mathbf{d})\\
    \Psi: \bbm \mathbf{a} & \mathbf{b} \\ \mathbf{c} & \mathbf{d}\ebm  &\mapsto&  h(\mathbf{a})v(\mathbf{a}\otimes\mathbf{c})u(\mathbf{a}^{-1}\otimes \mathbf{b}) .
    \eea
    This is an isomorphism.
    \een

\een

This completes the proof of Theorem~\ref{th:sl(2,2n)}. \hfill $\Box$

\subsection{Other groups of characteristic $2$}

We expect that Theorem~\ref{cons:classical} is mirrored by the following conjecture.

\begin{conjecture}\label{conjecture:char2}
Let $X$ be a black box  group encrypting one of the untwisted Chevalley groups ${\rm G}(2^n)$.  We assume that we  are given an involution $u \in X$.

Then there is a Monte-Carlo algorithm which constructs a polynomial time \textup{(}in $l(X)$\textup{)} isomorphism
\[
\Phi: {\rm G}(2^n) \longrightarrow X.
\]
The running time of the algorithm is polynomial in  $n$ and the Lie rank of ${\rm G}(2^n)$.
\end{conjecture}

As a comment to Conjecture~\ref{conjecture:char2}, we formulate here the following easy result.

\begin{theorem}\label{longroot} Let\/ $X$ be a black box group encrypting an untwisted Chevalley group ${\rm G}(2^n)$ {\rm (}with $n$ known{\rm )} and\/ $U < X$ an unipotent long root subgroup given as a black box subgroup of\/ $X$. Then there is a polynomial time, in $n$ and  Lie rank of ${\rm G}(2^n)$, Monte-Carlo algorithm which constructs a black box for $N_X(U)$ and a black box $\mathbb{U}$ for the field  $\GF_{2^n}$ interpreted in the action of $N_X(U)$ on $U$, with $U$ becoming the additive group of the field $\mathbb{U}$.
\end{theorem}

The proof of this theorem is an immediate and obvious generalization of Step~8 in the proof of Theorem~\ref{th:sl(2,2n)} in Section~\ref{sec:proofSl2odd}. Indeed,  it suffices to observe that $U$ is a TI-subgroup of $X$ (that is, $U\cap U^g =1$ or $U$ for all $g \in G$) and that all involutions in $U$ are conjugate in $N_X(U)$.

Theorem~\ref{longroot} suggests that structure recovery of black box Chevalley groups ${\rm G}(2^n)$ is likely to share some of the conceptual framework of Franz Timmesfeld's classification of groups generated by root type subgroups \cite{timmesfeld90.167,timmesfeld91.575,timmesfeld92.183,timmesfeld2001}. If so, then this will be strikingly similar to the use of Aschbacher's classical involutions \cite{aschbacher77.353,aschbacher80.411} and root $\sl_2$-subgroups in our structural theory of classical black box groups in odd characteristic \cite{suko03,suko12E,suko12A,suko12B,suko02,yalcinkaya07.171}.

\section*{Acknowledgements}

This paper would have never been written if the authors did not enjoy the warm hospitality offered to them at the Nesin Mathematics Village (in \c{S}irince, Izmir Province, Turkey) in August 2011, August 2012, and July 2013; our thanks go to Ali Nesin and to all volunteers and staff who have made the Village a mathematical paradise.

We thank Adrien Deloro for many fruitful discussions, in \c{S}irince and elsewhere, and
Bill Kantor and Rob Wilson for their helpful comments.

Special thanks go to our logician colleagues: Paola D'Aquino, Gregory Cherlin, Jan Kraj{\'{\i}}{\v{c}}ek,  Angus Macintyre, Jeff Paris, Jonathan Pila, and Alex Wilkie  for pointing to fascinating connections with logic and complexity theory.

We gratefully acknowledge the use of Paul Taylor's \emph{Commutative Diagrams} package, \url{http://www.paultaylor.eu/diagrams/}.

\nocite{yalcinkaya07.171}

\bibliographystyle{amsplain}
\bibliography{BlackBoxBibliography}

\end{document}